\documentclass[12pt,reqno]{article}
\usepackage{amsfonts}
\usepackage{amssymb,amscd,amsmath,amsthm,color}\usepackage{hyperref}
\newtheorem{theorem}{Theorem}
\newtheorem{lemma}[theorem]{Lemma}

\usepackage{enumerate,amssymb}

\newtheorem{conjecture}[theorem]{Conjecture}
\theoremstyle{remark}
\newtheorem{remark}[theorem]{Remark}

\numberwithin{equation}{section}

\newcommand{\tr}{\operatorname{tr}}

\begin{document}
\baselineskip=15pt

\title{An upper bound for the determinant of a diagonally balanced symmetric matrix}

\author{Minghua Lin}

\date{}

\maketitle

\begin{abstract}
   We prove a conjectured determinantal inequality: \begin{eqnarray*} \frac{\det J}{\prod_{i=1}^nJ_{ii}}\le 2\left(1-\frac{1}{n-1}\right)^{n-1},
\end{eqnarray*} where $J$ is a real $n\times n$ ($n\ge 2$) diagonally balanced symmetric matrix.

\end{abstract}

{\small\noindent

\noindent
AMS subjects classification 2010:  15A45.}

 \section{Introduction}
\noindent
  An $n\times n$ real matrix $J$ is (row) diagonally dominant if
\[\Delta_i(J) := |J_{ii}|-\sum_{j\ne i}|J_{ij}|\ge0, ~~\hbox{for}~~ i = 1,\ldots, n.\]
When $\Delta_i(J)=0$ for all $i$, we call such a matrix diagonally balanced.

In \cite{HLW12}, the following conjecture is made.

\begin{conjecture}
For a (entrywise) positive, diagonally balanced symmetric $J$, we have the bound:
\begin{eqnarray}\label{con}\frac{\det J}{\prod_{i=1}^nJ_{ii}}\le 2\left(1-\frac{1}{n-1}\right)^{n-1}\rightarrow\frac{2}{e}
\end{eqnarray}
\end{conjecture}

 Without loss of generality, we may assume $J_{ii}=1$ for all $i$. Then we can write $J=I_n+B$, where $B$ is a symmetric stochastic matrix with $B_{ii}=0$ for all $i$. Here $I_n$ means the identity matrix of size $n\times n$. A (row) stochastic matrix is a square matrix of nonnegative real numbers, with each row summing to $1$.  The trace of a square matrix $\cdot$ is denoted by $\tr \cdot$.

  The purpose of this short note is to give an affirmative answer to the conjecture.  The main result is the following theorem:
\begin{theorem}\label{thm1} Let $B$ be an $n\times n$ symmetric stochastic matrix with $B_{ii}=0$ for all $i$. Then
\begin{eqnarray}\label{main}\det (I_n+B)\le 2\left(1-\frac{1}{n-1}\right)^{n-1}.
\end{eqnarray}
The inequality is sharp.\end{theorem}

When $n=2$, (\ref{main}) is trivial, so in the sequel, we assume $n\ge 3$. The proof of Theorem \ref{thm1} is given in the next section.

 \section{Auxiliary results and the proofs}
\noindent
We start with some lemmas that are needed in the proof of our main result.

\begin{lemma}  \label{lem1} Let $B$ be an $n\times n$ symmetric stochastic matrix with $B_{ii}=0$ for all $i$. Then
\begin{eqnarray}\label{e21} \tr B^2\ge \frac{n}{n-1}. \end{eqnarray} The equality holds if and only if $B_{ij}=\frac{1}{n-1}$ for all $i\ne j$. \end{lemma}
\begin{proof} By the Cauchy-Schwarz inequality, \begin{eqnarray*} (n^2-n)\sum_{i\ne j}B_{ij}^2\ge \left(\sum_{i\ne j}B_{ij}\right)^2 =n^2,
\end{eqnarray*} so \[\tr B^2=\sum_{i\ne j}B_{ij}^2\ge\frac{n}{n-1}.\]
The equality case is trivial.
 \end{proof}

\begin{lemma} \label{lem2} Given $a>0$. Define $f(t)=(1+at)(1-t/a)^{a^2}$, $0\le t\le a$. Then $f(t)$ is decreasing.\end{lemma}
\begin{proof}  It suffices to show $\widetilde{f}(t)=\log f(t)$ is decreasing for $0< t<a$. Observing \[\widetilde{f}'(t)=\frac{a}{1+at}-\frac{a}{1-t/a}=-\frac{a(1+a^2)t}{(1+at)(a-t)}<0.\] The conclusion follows.
 \end{proof}

The key to the proof of Theorem \ref{thm1} is the following lemma.

\begin{lemma} (\cite{BSW82} or  \cite[Eq. (1.2)]{GJSW84}) \label{lem3} Let $A$ be an $n\times n$ positive semidefinite matrix.  If $m=\frac{\tr A}{n}$ and $s=\sqrt{\frac{\tr A^2}{n}- m^2}$, then
\begin{eqnarray}\label{e21}
(m - s\sqrt{n-1})(m + s/\sqrt{n-1})^{n-1}\le \det A\nonumber\\
\le (m + s\sqrt{n-1})(m - s/\sqrt{n-1})^{n-1}. \end{eqnarray} \end{lemma}

  {\it Proof of Theorem \ref{thm1}.}  Let $A=I_n+B$, then $A$ is positive semidefinite. Simple calculation gives $m:=\frac{\tr A}{n}=1$ and  $s^2:=\frac{\tr A^2}{n}- m^2=\frac{\tr B^2}{n}$. So by Lemma \ref{lem3}, we have
\begin{eqnarray}\label{e22} \det(I_n+B)\le  (1 + s\sqrt{n-1})(1- s/\sqrt{n-1})^{n-1},
\end{eqnarray} where $s=\sqrt{\frac{\tr B^2}{n}}$.
Note that $\tr B^2=\sum_{i\ne j}B_{ij}^2<n^2-n$ for $n\ge 3$, so $s<\sqrt{n-1}$. On the other hand, by  Lemma \ref{lem1}, we have $\frac{\tr B^2}{n}\ge \frac{1}{n-1}$, so $s\ge \frac{1}{\sqrt{n-1}}$. By Lemma \ref{lem2}, we know $f(s)=(1+s\sqrt{n-1})(1-s/\sqrt{n-1})^{n-1}$ is decreasing with respect to $s\in [\frac{1}{\sqrt{n-1}}, \sqrt{n-1})$. Thus \begin{eqnarray}\label{e23} f(s)\le f\left(\frac{1}{\sqrt{n-1}}\right)=2\left(1-\frac{1}{n-1}\right)^{n-1}.\end{eqnarray} Inequality (\ref{main}) now follows from  (\ref{e22}) and (\ref{e23}).

Taking  $B_{ij}=\frac{1}{n-1}$ for all $i\ne j$, the equality in (\ref{main}) holds. This proves the sharpness of (\ref{main}).

 \begin{remark} The lower bound of $\det A$ in (\ref{e21}) does not give a useful result for the lower bound for $\det(I_n+B)$ in Theorem \ref{thm1}. Indeed, define $g(s)=(1-s\sqrt{n-1})(1+s/\sqrt{n-1})^{n-1}$ for $s=\sqrt{\frac{\tr B^2}{n}}\ge \frac{1}{\sqrt{n-1}}$. In order $g(s)\ge 0$, we must require $s\le \frac{1}{\sqrt{n-1}}$. But now $g\left(\frac{1}{\sqrt{n-1}}\right)=0$.

  It is thus natural to ask whether there is a sharp lower bound $\Psi(n)$, depending on $n$ only, for $\det(I_n+B)$. Obviously, $\Psi(2)=0$, $\Psi(3)=\frac{1}{2}$.  \end{remark}

 \begin{remark} In the proof of Theorem \ref{thm1}, we do not require that the entries of $B$ to be positive. Thus Theorem \ref{thm1} is also valid for diagonally balance symmetric matrix $I_n+B$ with the entries of $B$ negative.\end{remark}

\vskip 10pt

\noindent
Minghua Lin

 Department of Applied Mathematics,

University of Waterloo,

 Waterloo, ON, N2L 3G1, Canada.

mlin87@ymail.com

\end{document}